\documentclass[12pt]{amsart}
\usepackage{amssymb,bbold}
\usepackage[all]{xy}

\theoremstyle{plain}
\newtheorem{theorem}{Theorem}

\newtheorem{lemma}[theorem]{Lemma}

\theoremstyle{definition}

\renewcommand{\le}{\leqslant}
\renewcommand{\ge}{\geqslant}

\begin{document}
\baselineskip 18pt

\title[A Note on Unbounded Convergences in Gremlin Projective Tensor Products]
      {A Note on Unbounded Convergences in Fremlin Projective Tensor Products}
\author[O.~Zabeti]{Omid Zabeti}
\address[O.~Zabeti]
  {Department of Mathematics, Faculty of Mathematics, Statistics, and Computer science,
   University of Sistan and Baluchestan, Zahedan,
   P.O. Box 98135-674. Iran}
\email{o.zabeti@gmail.com}
\maketitle

\begin{abstract}
We prove that the Fremlin projective tensor product of Banach lattices preserves unbounded norm convergence and unbounded absolute weak convergence of elementary tensors without any additional assumptions. The main tool is a tensor-lattice estimate controlling expressions of the form
\[
\lvert x_{\alpha}\otimes y_{\beta}-x\otimes y\rvert \wedge w,
\]
where $w$ is positive in the tensor product. This provides a direct and simplified proof of permanence results previously obtained under extra hypotheses.
\end{abstract}

\date{\today}

\maketitle
\section{Introduction}

Unbounded convergence has become an important tool in the study of Banach lattices and vector lattices. Among the most extensively investigated notions are unbounded norm convergence ($un$-convergence), unbounded absolute weak convergence ($uaw$-convergence), and unbounded order convergence ($uo$-convergence).

It is well known that both $un$-convergence and $uaw$-convergence are topological. More precisely, they are induced by suitable locally solid topologies. In contrast, $uo$-convergence is not topological in general; that is, $uo$-convergence on a vector lattice need not be generated by a topology.

In 1972, Fremlin \cite{Fremlin:72} constructed an Archimedean vector lattice tensor product
\[
E\overline{\otimes}F
\]
for Archimedean vector lattices $E$ and $F$. Subsequently, in 1974, he introduced the Fremlin projective tensor product
\[
E\widehat{\otimes}F
\]
for Banach lattices $E$ and $F$; see \cite{Fremlin:74}. These constructions provide natural lattice-theoretic extensions of the algebraic tensor product and have important applications in the theory of vector lattices and Banach lattices.

It is therefore natural to investigate the behavior of unbounded convergence under these tensor products. More precisely, suppose that
\[
x_{\alpha} \xrightarrow{un} x \quad \text{in } E
\qquad \text{and} \qquad
y_{\beta} \xrightarrow{un} y \quad \text{in } F.
\]
Does it then follow that
\[
x_{\alpha}\otimes y_{\beta} \xrightarrow{un} x\otimes y
\]
in $E\widehat{\otimes}F$? Similar questions can be raised for $uaw$-convergence and other unbounded modes of convergence.

In \cite{Z:25}, we proved that the Fremlin tensor product of Archimedean vector lattices behaves well with respect to $uo$-convergence. More precisely, we established that appropriate $uo$-convergence assumptions on nets in the factor spaces imply $uo$-convergence of the corresponding elementary tensors in the Fremlin tensor product. On the other hand, in \cite{Z:26}, we showed that, under some additional hypotheses, the Fremlin projective tensor product of Banach lattices preserves both unbounded norm convergence and unbounded absolute weak convergence.

The purpose of the present note is to show that the additional assumptions imposed in the above results are, in fact, unnecessary. The main ingredient is a simple but useful lemma which yields an estimate for expressions of the form
\[
\lvert x_{\alpha}\otimes y_{\beta}-x\otimes y\rvert\wedge w,
\]
where $w$ is a positive element of the relevant Fremlin tensor product. This estimate allows us to reduce the problem to appropriate wedge estimates in the factor spaces and consequently obtain assumption-free proofs of the main results.

More precisely, we prove that the Fremlin projective tensor product of Banach lattices preserves both unbounded norm convergence and unbounded absolute weak convergence of elementary tensors without any additional hypotheses. Our arguments are short and transparent and rely only on the lattice structure of the tensor product together with the aforementioned estimate. In addition, we give a simplified proof of the main result of \cite{Z:25}.

\section{Preliminaries}

In this section, we recall the notation, terminology, and basic facts that will be used throughout the paper. For standard background on vector lattices and Banach lattices, we refer the reader to \cite{AB1, AB}.

\subsection{Unbounded convergence}

Let $E$ be a vector lattice and let $(x_{\alpha})$ be a net in $E$.

The net $(x_{\alpha})$ is said to be \emph{unbounded order convergent}, or \emph{$uo$-convergent}, to $x\in E$ if
\[
|x_{\alpha}-x|\wedge u \xrightarrow{o} 0
\quad \text{for every } u\in E_{+}.
\]
Equivalently, for every $u\in E_{+}$, the net
\[
\bigl(|x_{\alpha}-x|\wedge u\bigr)
\]
is order convergent to zero.

Now, assume that $E$ is a Banach lattice. The net $(x_{\alpha})$ is said to be \emph{unbounded norm convergent}, or \emph{$un$-convergent}, to $x\in E$ if
\[
\bigl\|\,|x_{\alpha}-x|\wedge u\,\bigr\| \longrightarrow 0
\quad \text{for every } u\in E_{+}.
\]

Furthermore, $(x_{\alpha})$ is said to be \emph{unbounded absolutely weakly convergent}, or \emph{$uaw$-convergent}, to $x\in E$, denoted by
\[
x_{\alpha}\xrightarrow{uaw}x,
\]
if
\[
|x_{\alpha}-x|\wedge u \xrightarrow{w} 0
\quad \text{for every } u\in E_{+}.
\]
In other words, for every $u\in E_{+}$, the net
\[
\bigl(|x_{\alpha}-x|\wedge u\bigr)
\]
is weakly null.

For further information on unbounded convergence and related notions, we refer to
\cite{Den:17, GTX:17, KMT, Z:18}.

\subsection{Fremlin tensor products}

Let $E$ and $F$ be Archimedean vector lattices. Fremlin constructed in \cite{Fremlin:72} an Archimedean vector lattice tensor product, denoted by
\[
E\overline{\otimes}F,
\]
which contains the algebraic tensor product $E\otimes F$ as a vector subspace.

Now assume that $E$ and $F$ are Banach lattices. In \cite{Fremlin:74}, Fremlin introduced the projective tensor product
\[
E\widehat{\otimes}F,
\]
which is a Banach lattice. This space is the completion of the algebraic tensor product $E\otimes F$ with respect to the projective norm $\|\cdot\|_{|\pi|}$. The norm $\|\cdot\|_{|\pi|}$ is a lattice cross-norm and satisfies
\[
\|x\otimes y\|_{|\pi|}
=
\|x\|\,\|y\|
\quad \text{for all } x\in E \text{ and } y\in F.
\]
Moreover, $E\overline{\otimes}F$ is a norm-dense vector sublattice of
$E\widehat{\otimes}F$.

\section{main results}

The following lemma is of independent interest and plays a crucial role in the subsequent results of this note.
\begin{lemma}\label{new1}
Let $E$ and $F$ be vector lattices. For any $x, a \in E_+$ and $y, b \in F_+$, the following inequality holds in the Fremlin tensor product $E \overline{\otimes}F$:
\begin{equation}
(x \otimes y) \wedge (a \otimes b) \le (x \wedge a) \otimes b + a \otimes (y \wedge b)
\end{equation}
\end{lemma}

\begin{proof}
Let $I_a$ and $I_b$ denote the principal ideals generated by $a$ in $E$ and $b$ in $F$, respectively. By Kakutani's representation theorem, there exist compact Hausdorff spaces $K_1$ and $K_2$ and lattice isomorphisms:
\[
\phi_1: I_a \to C(K_1) \quad \text{and} \quad \phi_2: I_b \to C(K_2)
\]
such that $\phi_1(a) = \mathbf{1}_{K_1}$ and $\phi_2(b) = \mathbf{1}_{K_2}$, where $\mathbf{1}$ denotes the constantly one function. 

Note that we can consider the Fremlin tensor product $C(K_1)\overline{\otimes} C(K_2)$ as an order and norm dense sublattice of $C(K_1\times K_2)$. Under this representation, the elementary tensor $u \otimes v$ corresponds to the function $(u \otimes v)(s, t) = u(s)v(t)$ for all $s \in K_1$ and $t \in K_2$.

To establish the inequality (1), it suffices to show that the corresponding pointwise inequality holds for the functions representing these elements on $K_1 \times K_2$. That is, for every $s \in K_1$ and $t \in K_2$, we must show:
\begin{equation}
\min\{x(s)y(t), a(s)b(t)\} \le \min\{x(s), a(s)\}b(t) + a(s)\min\{y(t), b(t)\}
\end{equation}
Note that since $x, a, y, b$ are positive elements, their representatives are non-negative functions, i.e., $x(s), a(s), y(t), b(t) \ge 0$. We evaluate inequality (2) by partitioning the domain $K_1 \times K_2$ into disjoint cases based on the order relations of the coordinates:

\subsection*{Case 1: $x(s) \ge a(s)$}
In this case, we have $\min\{x(s), a(s)\} = a(s)$. The right-hand side (RHS) of (2) simplifies to:
\[
\text{RHS} = a(s)b(t) + a(s)\min\{y(t), b(t)\}
\]
For the left-hand side (LHS) of (2), by the definition of the minimum, we have:
\[
\text{LHS} = \min\{x(s)y(t), a(s)b(t)\} \le a(s)b(t)
\]
Since $a(s) \ge 0$ and $\min\{y(t), b(t)\} \ge 0$, it follows that $a(s)\min\{y(t), b(t)\} \ge 0$. Therefore:
\[
\text{LHS} \le a(s)b(t) \le a(s)b(t) + a(s)\min\{y(t), b(t)\} = \text{RHS}
\]
which proves the inequality for this case.

\subsection*{Case 2: $x(s) < a(s)$}
In this case, we have $\min\{x(s), a(s)\} = x(s)$. The RHS of (2) becomes:
\[
\text{RHS} = x(s)b(t) + a(s)\min\{y(t), b(t)\}
\]
We now partition this case into two sub-cases depending on the relation between $y(t)$ and $b(t)$:

\subsubsection*{Sub-case 2.1: $y(t) \ge b(t)$}
Here, $\min\{y(t), b(t)\} = b(t)$. The RHS simplifies to:
\[
\text{RHS} = x(s)b(t) + a(s)b(t)
\]
For the LHS, we observe:
\[
\text{LHS} = \min\{x(s)y(t), a(s)b(t)\} \le a(s)b(t)
\]
Since $x(s) \ge 0$ and $b(t) \ge 0$, we have $x(s)b(t) \ge 0$. Thus:
\[
\text{LHS} \le a(s)b(t) \le x(s)b(t) + a(s)b(t) = \text{RHS}
\]
which satisfies the inequality.

\subsubsection*{Sub-case 2.2: $y(t) < b(t)$}
Here, $\min\{y(t), b(t)\} = y(t)$. The RHS simplifies to:
\[
\text{RHS} = x(s)b(t) + a(s)y(t)
\]
For the LHS, we have:
\[
\text{LHS} = \min\{x(s)y(t), a(s)b(t)\} \le x(s)y(t)
\]
Since $x(s) \ge 0$ and $y(t) < b(t)$, we obtain $x(s)y(t) \le x(s)b(t)$. Furthermore, since $a(s) \ge 0$ and $y(t) \ge 0$, the term $a(s)y(t)$ is non-negative. Combining these yields:
\[
\text{LHS} \le x(s)y(t) \le x(s)b(t) \le x(s)b(t) + a(s)y(t) = \text{RHS}
\]
This completes the verification for all possible cases.

\medskip
Since the inequality (2) holds pointwise for all $(s, t) \in K_1 \times K_2$, the lattice isomorphism preserves the order structure, implying that:
\[
(x \otimes y) \wedge (a \otimes b) \le (x \wedge a) \otimes b + a \otimes (y \wedge b)
\]
holds in $I_a \overline{\otimes} I_b$. Since $I_a \overline{\otimes} I_b$ is a vector sublattice of $E \overline{\otimes} F$, the inequality holds globally in $E \overline{\otimes} F$.
\end{proof}
Now we present a simple proof of the main result of \cite{Z:25}
(Theorem~7).

\begin{theorem}\label{new}
Suppose that $E$ and $F$ are Archimedean vector lattices.
If $x_{\alpha}\xrightarrow{uo}x$ in $E$ and
$y_{\beta}\xrightarrow{uo}y$ in $F$, then
\[
x_{\alpha}\otimes y_{\beta}\xrightarrow{uo}x\otimes y
\quad \text{in } E\overline{\otimes}F .
\]
\end{theorem}

\begin{proof}
Let $w\in (E\overline{\otimes}F)_{+}$ be arbitrary.
By \cite[1A(d)]{Fremlin:74}, there exist elements
$x_0\in E_{+}$ and $y_0\in F_{+}$ such that
\[
w\le x_0\otimes y_0.
\]
Using Lemma~\ref{new1}, we obtain
\begin{align*}
\lvert x_{\alpha} \otimes y_{\beta} - x \otimes y \rvert \wedge w 
&\le \lvert x_{\alpha} \otimes y_{\beta} - x \otimes y \rvert \wedge (x_0 \otimes y_0) \\
&= \lvert x_{\alpha} \otimes (y_{\beta} - y) + (x_{\alpha} - x) \otimes y \rvert \wedge (x_0 \otimes y_0) \\
&\le \bigl(\lvert x_{\alpha} \rvert \otimes \lvert y_{\beta} - y \rvert\bigr) \wedge (x_0 \otimes y_0)
   + \bigl(\lvert x_{\alpha} - x \rvert \otimes \lvert y \rvert\bigr) \wedge (x_0 \otimes y_0) \\
&\le \bigl(\lvert x_{\alpha} - x \rvert \otimes \lvert y_{\beta} - y \rvert\bigr) \wedge (x_0 \otimes y_0)
   + \bigl(\lvert x \rvert \otimes \lvert y_{\beta} - y \rvert\bigr) \wedge (x_0 \otimes y_0) \\
&\quad + \bigl(\lvert x_{\alpha} - x \rvert \otimes \lvert y \rvert\bigr) \wedge (x_0 \otimes y_0) \\
&\le \bigl(\lvert x_{\alpha} - x \rvert \wedge x_0\bigr) \otimes y_0
   + x_0 \otimes \bigl(\lvert y_{\beta} - y \rvert \wedge y_0\bigr) \\
&\quad + \bigl(\lvert x \rvert \vee x_0\bigr) \otimes \bigl(\lvert y_{\beta} - y \rvert \wedge y_0\bigr)
   + \bigl(\lvert x_{\alpha} - x \rvert \wedge x_0\bigr) \otimes \bigl(\lvert y \rvert \vee y_0\bigr).
\end{align*}

The right-hand expressions are order null by using the assumption and also \cite[Proposition 8]{Z:25}. This shows that $x_{\alpha}\otimes y_{\beta}\xrightarrow{uo}x\otimes y$ in $E\overline{\otimes}F$. This would complete the proof. 

\end{proof}

\begin{theorem}\label{un}
Suppose that $E$ and $F$ are Banach lattices.
If $x_{\alpha}\xrightarrow{un}x$ in $E$ and
$y_{\beta}\xrightarrow{un}y$ in $F$, then
\[
x_{\alpha}\otimes y_{\beta}\xrightarrow{un}x\otimes y
\quad \text{in } E\widehat{\otimes}F .
\]
\end{theorem}

\begin{proof}
By \cite[Lemma 3.4]{Tay:19}, it is enough to prove that the convergence happens in $E\overline{\otimes}F$. Let $w\in (E\overline{\otimes}F)_{+}$ be arbitrary.
By \cite[1A(d)]{Fremlin:74}, there exist elements
$x_0\in E_{+}$ and $y_0\in F_{+}$ such that
\[
w\le x_0\otimes y_0.
\]
Using Lemma~\ref{new1}, we obtain
\begin{align*}
\lvert x_{\alpha} \otimes y_{\beta} - x \otimes y \rvert \wedge w 
&\le \lvert x_{\alpha} \otimes y_{\beta} - x \otimes y \rvert \wedge (x_0 \otimes y_0) \\
&= \lvert x_{\alpha} \otimes (y_{\beta} - y) + (x_{\alpha} - x) \otimes y \rvert \wedge (x_0 \otimes y_0) \\
&\le \bigl(\lvert x_{\alpha} \rvert \otimes \lvert y_{\beta} - y \rvert\bigr) \wedge (x_0 \otimes y_0)
   + \bigl(\lvert x_{\alpha} - x \rvert \otimes \lvert y \rvert\bigr) \wedge (x_0 \otimes y_0) \\
&\le \bigl(\lvert x_{\alpha} - x \rvert \otimes \lvert y_{\beta} - y \rvert\bigr) \wedge (x_0 \otimes y_0)
   + \bigl(\lvert x \rvert \otimes \lvert y_{\beta} - y \rvert\bigr) \wedge (x_0 \otimes y_0) \\
&\quad + \bigl(\lvert x_{\alpha} - x \rvert \otimes \lvert y \rvert\bigr) \wedge (x_0 \otimes y_0) \\
&\le \bigl(\lvert x_{\alpha} - x \rvert \wedge x_0\bigr) \otimes y_0
   + x_0 \otimes \bigl(\lvert y_{\beta} - y \rvert \wedge y_0\bigr) \\
&\quad + \bigl(\lvert x \rvert \vee x_0\bigr) \otimes \bigl(\lvert y_{\beta} - y \rvert \wedge y_0\bigr)
   + \bigl(\lvert x_{\alpha} - x \rvert \wedge x_0\bigr) \otimes \bigl(\lvert y \rvert \vee y_0\bigr).
\end{align*}
So that 
\begin{align*}
\bigl\|\,|x_{\alpha}\otimes y_{\beta}-x\otimes y|\wedge w\,\bigr\|
&\le
\bigl\|\,|x_{\alpha}-x|\wedge x_0\,\bigr\|\,\|y_0\| \\
&\quad +
\|x_0\|\,\bigl\|\,|y_{\beta}-y|\wedge y_0\,\bigr\| \\
&\quad +
\bigl\|x_0\vee |x|\bigr\|\,
\bigl\|\,|y_{\beta}-y|\wedge y_0\,\bigr\| \\
&\quad +
\bigl\|\,|x_{\alpha}-x|\wedge x_0\,\bigr\|\,
\bigl\|y_0\vee |y|\bigr\|.
\end{align*}

The right-hand expressions are norm null by using the assumption. This shows that $x_{\alpha}\otimes y_{\beta}\xrightarrow{un}x\otimes y$ in $E\overline{\otimes}F$ so that in $E\widehat{\otimes}F$. This would complete the proof. 

\end{proof}

\begin{theorem}\label{uaw}
Suppose that $E$ and $F$ are Banach lattices.
If $x_{\alpha}\xrightarrow{uaw}x$ in $E$ and
$y_{\beta}\xrightarrow{uaw}y$ in $F$, then
\[
x_{\alpha}\otimes y_{\beta}\xrightarrow{uaw}x\otimes y
\quad \text{in } E\widehat{\otimes}F .
\]
\end{theorem}

\begin{proof}
By \cite[Lemma 3.4]{Tay:19}, it is enough to prove that the convergence happens in $E\overline{\otimes}F$. Let $w\in (E\overline{\otimes}F)_{+}$ be arbitrary.
By \cite[1A(d)]{Fremlin:74}, there exist elements
$x_0\in E_{+}$ and $y_0\in F_{+}$ such that
\[
w\le x_0\otimes y_0.
\]
Using Lemma~\ref{new1}, we obtain
\begin{align*}
\lvert x_{\alpha} \otimes y_{\beta} - x \otimes y \rvert \wedge w 
&\le \lvert x_{\alpha} \otimes y_{\beta} - x \otimes y \rvert \wedge (x_0 \otimes y_0) \\
&= \lvert x_{\alpha} \otimes (y_{\beta} - y) + (x_{\alpha} - x) \otimes y \rvert \wedge (x_0 \otimes y_0) \\
&\le \bigl(\lvert x_{\alpha} \rvert \otimes \lvert y_{\beta} - y \rvert\bigr) \wedge (x_0 \otimes y_0)
   + \bigl(\lvert x_{\alpha} - x \rvert \otimes \lvert y \rvert\bigr) \wedge (x_0 \otimes y_0) \\
&\le \bigl(\lvert x_{\alpha} - x \rvert \otimes \lvert y_{\beta} - y \rvert\bigr) \wedge (x_0 \otimes y_0)
   + \bigl(\lvert x \rvert \otimes \lvert y_{\beta} - y \rvert\bigr) \wedge (x_0 \otimes y_0) \\
&\quad + \bigl(\lvert x_{\alpha} - x \rvert \otimes \lvert y \rvert\bigr) \wedge (x_0 \otimes y_0) \\
&\le \bigl(\lvert x_{\alpha} - x \rvert \wedge x_0\bigr) \otimes y_0
   + x_0 \otimes \bigl(\lvert y_{\beta} - y \rvert \wedge y_0\bigr) \\
&\quad + \bigl(\lvert x \rvert \vee x_0\bigr) \otimes \bigl(\lvert y_{\beta} - y \rvert \wedge y_0\bigr)
   + \bigl(\lvert x_{\alpha} - x \rvert \wedge x_0\bigr) \otimes \bigl(\lvert y \rvert \vee y_0\bigr).
\end{align*}
Recall that $(E\widehat{\otimes}F)'_{+}=B^{r}(E\times F)_{+}$, where $B^{r}(E\times F)$ denotes the Banach lattice of all regular bounded bilinear forms on $E\times F$.
Under this identification, each $\widetilde{f}\in (E\widehat{\otimes}F)'_{+}$ corresponds to some $f\in B^{r}(E\times F)_{+}$ such that
\[
\widetilde{f}(x\otimes y)=f(x,y), \qquad x\in E,\; y \in F,
\]
see \cite[5I]{Fremlin:74} for more details. 

Take any $\widetilde{f}\in (E\widehat{\otimes}F)'_{+}$. We have

\begin{align*}
\widetilde{f}(|x_{\alpha}\otimes y_{\beta}-x\otimes y|\wedge w)
&\le
\widetilde{f}((|x_{\alpha}-x|\wedge x_0)\otimes y_0) \\
&\quad +
\widetilde{f}(x_0\otimes (|y_{\beta}-y|\wedge y_0))\\
&\quad +
\widetilde{f}((x_0\vee |x|)\otimes (|y_{\beta}-y|\wedge y_0))\\
&\quad +
\widetilde{f}(|x_{\alpha}-x|\wedge x_0\otimes (y_0\vee |y|)).
\end{align*}

The right-hand expressions are norm null by using the assumption. Since the mappings $f_{x_{0}\vee |x|}:F\to\mathbb{R}$ and $f_{y_{0}\vee |y|}:E\to\mathbb{R}$ are positive linear functionals, both terms converge to zero.
This shows that $x_{\alpha}\otimes y_{\beta}\xrightarrow{uaw}x\otimes y$ in $E\overline{\otimes}F$ so that in $E\widehat{\otimes}F$. This would complete the proof. 

\end{proof}

\textbf{The author admits that there is no conflict of interest in this paper and he wrote the paper, thoroughly}. 


\begin{thebibliography}{1}
\bibitem{AB1}
C.~D.~Aliprantis and O.~Burkinshaw,
\emph{Locally Solid Riesz Spaces with Applications to Economics},
Mathematical Surveys and Monographs, Vol.~105,
American Mathematical Society, Providence, RI, 2003.

\bibitem{AB}
C.~D.~Aliprantis and O.~Burkinshaw,
\emph{Positive Operators},
Springer, Dordrecht, 2006.

\bibitem{Den:17}
Y.~Deng, M.~O'Brien, and V.~G.~Troitsky,
Unbounded norm convergence in Banach lattices,
\emph{Positivity} \textbf{21} (2017), no.~3, 963--974.

\bibitem{Fremlin:72}
D.~H.~Fremlin,
Tensor products of Archimedean vector lattices,
\emph{Amer.\ J.\ Math.} \textbf{94} (1972), 777--798.

\bibitem{Fremlin:74}
D.~H.~Fremlin,
Tensor products of Banach lattices,
\emph{Math.\ Ann.} \textbf{211} (1974), 87--106.

\bibitem{GTX:17}
N.~Gao, V.~G.~Troitsky, and F.~Xanthos,
Unbounded order convergence and applications to Ces\`aro means in Banach lattices,
\emph{Israel J.\ Math.} \textbf{220} (2017), 649--689.

\bibitem{KMT}
M.~Kandi\'{c}, M.~A.~A.~Marabeh, and V.~G.~Troitsky,
Unbounded norm topology in Banach lattices,
\emph{J.\ Math.\ Anal.\ Appl.} \textbf{451} (2017), no.~1, 259--279.


\bibitem{Tay:19}
M.~A.~Taylor,
Unbounded topologies and uo-convergence in locally solid vector lattices,
\emph{J.\ Math.\ Anal.\ Appl.} \textbf{472} (2019), no.~1, 981--1000.



\bibitem{Z:18}
O.~Zabeti,
Unbounded absolute weak convergence in Banach lattices,
\emph{Positivity} \textbf{22} (2018), no.~1, 501--505.

\bibitem{Z:26}
O.~Zabeti,
Fremlin tensor  product of Banach lattices and
unbounded convergences,
Submitted.

\bibitem{Z:25}
O.~Zabeti,
Fremlin tensor product behaves well with the unbounded order convergence,
\emph{Acta Sci.\ Math.\ (Szeged)} (2025),
doi:10.1007/s44146-025-00183-9.

  \end{thebibliography}
\end{document}